\documentclass[11pt,reqno]{amsart}
\usepackage{amsmath,amsthm,amssymb,mathrsfs}
\usepackage[tmargin=1.15in,bmargin=1.15in,rmargin=1.2in,lmargin=1.2in]{geometry}
\usepackage[breaklinks=true]{hyperref}
\allowdisplaybreaks

\theoremstyle{plain}
\newtheorem{theorem}{Theorem}[section]

\renewcommand{\subset}{\subseteq}

\newtheorem{cor}[theorem]{Corollary}

\newtheorem{proposition}[theorem]{Proposition}
\newtheorem{lemma}[theorem]{Lemma}

\theoremstyle{definition}
\newtheorem{definition}[theorem]{Definition}
\newtheorem{remark}[theorem]{Remark}
\newtheorem{example}[theorem]{Example}

\numberwithin{equation}{section}

\newcommand{\bm}{\mathbf{m}}

\newcommand{\bu}{\mathbf{u}}
\newcommand{\bv}{\mathbf{v}}

\newcommand\R{\mathbb{R}}
\newcommand\Z{\mathbb{Z}}
\newcommand\F{\mathbb{F}}

\renewcommand{\geq}{\geqslant}
\renewcommand{\leq}{\leqslant}

\newcommand\perm {\mathop{\rm perm}}

\begin{document}

\title[From Cauchy's determinant formula to bosonic/fermionic immanant
identities]{From Cauchy's determinant formula to\\
bosonic and fermionic immanant identities}

\author{Apoorva Khare}
\address[A.~Khare]{Department of Mathematics, Indian Institute of
Science, Bangalore -- 560012, India; and Analysis and Probability
Research Group, Bangalore -- 560012, India}
\email{\tt khare@iisc.ac.in}

\author{Siddhartha Sahi}
\address[S.~Sahi]{Department of Mathematics, Rutgers University,
Piscataway 08854, USA}
\email{\tt sahi@math.rutgers.edu}

\date{\today}

\keywords{Schur polynomial, symmetric group, character, immanant,
determinantal identity, signed permutations, 
symmetric functions, bosonic variables, fermionic variables}

\subjclass[2010]{Primary 05E05; 
Secondary 15A15, 
15A24, 
20C15 
}

\begin{abstract}
Cauchy's determinant formula (1841) involving $\det ((1-u_i v_j)^{-1})$
is a fundamental result in symmetric function theory. It has been
extended in several directions, including a determinantal extension by
Frobenius [\textit{J.\ reine angew.\ Math.}\ 1882] involving a sum of two
geometric series in $u_i v_j$. This theme also resurfaced in a matrix
analysis setting in a paper by Horn
[\textit{Trans.\ Amer.\ Math.\ Soc.}\ 1969] --
where the computations are attributed to Loewner -- and in recent works
by Belton--Guillot--Khare--Putinar [\textit{Adv.\ Math.}\ 2016] and
Khare--Tao [\textit{Amer.\ J.\ Math.}\ 2021]. These formulas were
recently unified and extended in [\textit{Trans.\ Amer.\ Math.\ Soc.}\
2022] to arbitrary power series, with commuting/bosonic variables
$u_i, v_j$.

In this note we formulate analogous permanent identities, and in fact,
explain how all of these results are a special case of a more general
identity, for any character -- in fact, any complex class function -- of
any finite group that acts on the bosonic variables $u_i$ and on the
$v_j$ via signed permutations. (We explain why larger linear groups do
not work, via a -- perhaps novel -- ``symmetric function''
characterization of signed permutation matrices that holds over any
integral domain.) We then provide fermionic analogues of these formulas,
as well as of the closely related Cauchy product identities.
\end{abstract}

\maketitle

\section{Introduction}

\subsection{Post-1960 results: Entrywise positivity preservers and Schur
polynomials}\label{Smodern}

The goal of this note is to extend some classical and modern symmetric
function determinantal identities to other characters of the symmetric
group (and its subgroups), and then to formulate and show fermionic
counterparts of these. The origins of this work lie in classical
identities by Cauchy and Frobenius, but also in a computation -- see
Theorem~\ref{Tloewner} -- that originally appears in a letter by Charles
Loewner to Josephine Mitchell on October 24, 1967 (as observed by the
first-named author in the Stanford Library archives). Subsequently, this
computation, and the broader result on ``entrywise functions,'' appeared
in print in the thesis of Loewner's PhD student, Roger Horn -- see also
the proof of \cite[Theorem 1.2]{horn}, which Horn attributes to Loewner.

In his letter, Loewner explained that he was interested in understanding
functions acting \textit{entrywise} on positive semidefinite matrices
(i.e., real symmetric matrices with non-negative eigenvalues) of a fixed
size, and preserving positivity. Previously, results by Schur,
Schoenberg, and Rudin had classified the dimension-free preservers, i.e.,
the entrywise maps preserving positivity in \textit{all} dimensions
\cite{Schur1911, Schoenberg42, Rudin59}. In contrast, in a fixed
dimension $d$, such a classification remains open to date, even for
$d=3$; moreover, Loewner's 1967 result is still state-of-the-art, in that
it is (essentially) the only known necessary condition for a general
entrywise function preserving positivity in a fixed dimension. We refer
the reader to e.g.\ \cite{horndet} for more details.

The present work begins by isolating from Loewner's positivity/analysis
result, the following algebraic calculation. Fix an integer $n \geq 2$;
given a matrix $A = (a_{ij})$, here and below $f[A]$ denotes the matrix
with $(i,j)$-entry $f(a_{ij})$.

\begin{theorem}[Loewner]\label{Tloewner}
Suppose $f : \R \to \R$ is a smooth function, $n \geq 2$, and let
$\bu = (u_1, \dots, u_n)^T \in \R^n$. Define the determinant function
\[
\Delta : \R \to \R, \qquad t \mapsto \det (f(t u_i u_j))_{i,j=1}^n = \det
f[ t \bu \bu^T ].
\]
Then $\Delta(0) = \cdots = \Delta^{\binom{n}{2} - 1}(0) = 0$, and the
next derivative is
\begin{equation}\label{Eloewner}
\Delta^{\binom{n}{2}}(0) = \binom{\binom{n}{2}}{1,2,\dots,n-1}
\prod_{i<j} (u_j - u_i)^2 \cdot f(0) f'(0) \cdots f^{(n-1)}(0).
\end{equation}
\end{theorem}

In particular, if $f(t)$ is a convergent power series $\sum_{n \geq 0}
f_n t^n$, then within a suitable radius of convergence,
\[
\Delta(t) = t^{\binom{n}{2}} \prod_{i<j} (u_j - u_i)^2 \cdot f_0 f_1
\cdots f_{n-1} + \text{ higher order terms}.
\]

The first term on the right-hand side of Equation~\eqref{Eloewner} is a
multinomial coefficient, and the reader will recognize the next product
as the square of a Vandermonde determinant for the matrix with entries
$u_i^{n-j}, \ 1 \leq i,j \leq n$. What the reader may find harder to
recognize is that Equation~\eqref{Eloewner} contains a ``hidden'' Schur
polynomial (these are defined presently) in the variables $u_i$: the
simplest of them all, $s_{(0,\dots,0)}(\bu) = 1$. In particular,
if one goes even one derivative beyond Loewner's stopping point, one
immediately uncovers other, nontrivial Schur polynomials. This is stated
precisely in Theorem~\ref{Tsymm}.

The presence of the lurking (simplest) Schur polynomial
in~\eqref{Eloewner} was suspected owing to very recent sequels to
Loewner's matrix positivity result. First with Belton--Guillot--Putinar
\cite{BGKP-fixeddim} and then with Tao \cite{KT}, the first-named author
found (the first) examples of polynomial maps with at least one negative
coefficient, which preserve positivity in a fixed dimension when applied
entrywise. These papers uncovered novel connections between polynomials
that entrywise preserve positivity and Schur polynomials, and in
particular, obtained expansions for $\det f[t \bu \bv^T]$ in terms of
Schur polynomials, for all polynomials $f(t)$. This suggested revisiting
the general case due to Loewner (in slightly greater generality, as
above: for $\det f[t \bu \bv^T]$).

\subsection{Pre-1900 results: Cauchy and Frobenius}

We now go back in history and remind the reader of the first such
determinantal identities involving Schur polynomials.
Recall the well-known Cauchy determinant identity \cite{Cauchy},
\cite[Chapter I.4, Example 6]{Macdonald}: if $B$ is the $n \times n$
matrix with entries $(1 - u_i v_j)^{-1} = \sum_{M \geq 0} (u_i v_j)^M$
for variables $u_i, v_j$ with $1 \leq i, j \leq n$, then
\begin{equation}\label{Ecauchy}
\det B = V(\bu) V(\bv) \sum_\bm s_\bm(\bu) s_\bm(\bv),
\end{equation}

\noindent where $V(\bu)$ for a finite tuple $\bu = (u_i)_{i \geq 1}$
denotes the ``Vandermonde determinant'' $\prod_{i < j} (u_j - u_i)$,
and the sum runs over all partitions $\bm$ with at most $n$ parts. Here,
a partition $\bm = (m_1, \dots, m_n)$ simply means a weakly decreasing
sequence of nonnegative integers $m_1 \geq \cdots \geq m_n \geq 0$; and
we use Cauchy's definition \cite{Jacobi} for the Schur polynomial
$s_\bm(\bv)$, namely,
\[
s_\bm(v_1, \dots, v_n) :=
\frac{\det (v_j^{m_i + n-i})_{i,j=1}^n}{\det (v_j^{n-i})_{i,j=1}^n}.
\]
(This definition differs from that in the literature, e.g.\ in
\cite{Macdonald}.) Here and below, we restrict to $n$ arguments $v_j$,
to go with the $n$ exponents $m_i$.

See also \cite[Section 5]{Ku} and the references therein, as well as
\cite{IIO, IOTZ, Kr1, Kr2, LLT, Macdonald} for other determinantal
identities involving symmetric functions.

As discussed in Section~\ref{Smodern}, in this paper we focus on the
specific form of the determinant in~\eqref{Ecauchy}, i.e.\ where one
applies to all $u_i v_j$ some power series (Equation~\eqref{Ecauchy}
considers the case of $f(x) = 1/(1-x) = \sum_{M \geq 0} x^M$), and then
computes the determinant. For instance, if $f(x)$ has fewer than $n$
monomials then $f[ \bu \bv^T ]$ is a sum of fewer than $n$ rank-one
matrices, hence is singular. (For more general polynomials -- as
mentioned above -- the formula was worked out in \cite{KT}.) Another such
formula was shown by Frobenius \cite{Fr}, in fact in greater
generality.\footnote{Here one uses theta functions and obtains elliptic
Frobenius--Stickelberger--Cauchy determinant (type) identities; see also
\cite{Amd,FS}.} The formula appears in Rosengren--Schlosser
\cite[Corollary 4.7]{RS} as well, as a consequence of their Theorem~4.4;
and it implies a more general determinantal identity
than~\eqref{Ecauchy}, with $(1-cx)/(1-x)$ replacing $1/(1-x)$
and the sum again running over all partitions with at most $n$ parts:
\begin{align}\label{Eschlosser}
&\ \det \left( \frac{1 - c u_i v_j}{1 - u_i v_j} \right)_{i,j=1}^n\\
= &\ V(\bu) V(\bv) (1-c)^{n-1} \left( \sum_{\bm\; :\; m_n = 0} s_\bm(\bu)
s_\bm(\bv) + (1-c) \sum_{\bm\; :\; m_n > 0} s_\bm(\bu) s_\bm(\bv)
\right).
\notag
\end{align}

\subsection{The present work}

Given the many precursors listed above, it is natural to seek a more
general identity, i.e.\ the expansion of $\det f[\bu \bv^T]$, where
$f[\bu \bv^T]$ is the entrywise application of an arbitrary (formal)
power series $f$ to the rank-one matrix $\bu \bv^T = (u_i
v_j)_{i,j=1}^n$. This question was recently answered by the first-named
author -- including additional special cases -- again in the context of
matrix positivity preservers. 

\begin{theorem}[Khare, \cite{horndet}]\label{Tsymm}
Fix a unital commutative ring $R$ and let $t$ be an indeterminate. Let
$f(t) := \sum_{M \geq 0} f_M t^M \in R[[t]]$ be an arbitrary formal power
series. Given vectors $\bu, \bv \in R^n$ for some $n \geq 1$, we have:
\begin{equation}\label{Esymm}
\det f[ t \bu \bv^T ] = V(\bu) V(\bv) \sum_{M \geq 0} t^{M +
\binom{n}{2}}  \sum_{\bm = (m_1, \dots, m_n) \; \vdash M} s_\bm(\bu)
s_\bm(\bv) \prod_{i=1}^n f_{m_i + n-i},
\end{equation}
where $\bm \vdash M$ means that $\bm$ is a partition whose components sum
to $M$.
\end{theorem}

The goal of this short note is to show that these identities hold more
generally -- not just for determinants, but also e.g.\ for permanents,
where
\[
\perm(A_{n \times n}) := \sum_{\sigma \in S_n} a_{1 \sigma(1)} \cdots
a_{n \sigma(n)}.
\]
Thus we show below:

\begin{theorem}\label{Tperm}
With notation as in Theorem~\ref{Tsymm} (and over any commutative unital
ring $R$), we have:
\[
n! \perm f[ t \bu \bv^T ] = \sum_{{\bf m} \in
\Z_{\geq 0}^n} t^{m_1 + \cdots + m_n} \perm(\bu^{\circ {\bf m}})
\perm(\bv^{\circ {\bf m}}) \prod_{i=1}^n f_{m_i},
\]
where $\bv^{\circ {\bf m}} := (v_j^{m_i})$ (and similarly
for $\bu^{\circ {\bf m}}$), and ${\bf m} \geq {\bf 0}$ is interpreted
coordinatewise.
\end{theorem}

We show this result as well as Theorem~\ref{Tsymm} by a common
proof. In fact we go beyond permanents: we provide such an
identity for an arbitrary character of an arbitrary subgroup $G$ of $S_n$
-- and even of the hyperoctahedral group: $G \leq S_2 \wr S_n$.
Thus, our proof differs from the approach in \cite{horndet}, and proceeds
via group representation theory. We then produce a fermionic analogue of
the bosonic immanant ``master identity,'' in which the variables $u_i$
anti-commute, as do the $v_j$. For quick reference, these identities are
summarized in the following table.

{\renewcommand{\arraystretch}{1.3}
\begin{table}[h]
\begin{tabular}{|c|c|c|}
\hline
& Even (bosonic) variables & Odd (fermionic) variables\\
\hline
Determinant (for $S_n$) & \eqref{Edet} (see \cite{horndet}) &
\eqref{Edetodd} \\
\hline
Permanent (for $S_n$) & \eqref{Eperm} & \eqref{Epermodd} \\
\hline
Arbitrary immanants & & \\
for subgroups of $S_2 \wr S_n$ & \eqref{Egeneral0} &
\eqref{Egeneral-odd} \\
\hline
(Bi)Product identities & \eqref{Eprodeven} (see e.g.\ \cite{Macdonald}) &
\eqref{Eprododd} \\
\hline
\end{tabular}\bigskip
\caption{The first three rows provide formulas for an arbitrary formal
power series applied entrywise to the matrix $t \bu \bv = (t u_i
v_j)_{i,j=1}^n$. The fourth row computes the product of $(1 - u_i
v_j)^{-1}$ or of $(1 + u_i v_j)$. Two of these formulas can be found in
earlier literature, see \cite{horndet,Macdonald}.}
\end{table}}

\section{Immanant identities for bosonic variables}

\subsection{Establishing the setting}

In order to state and prove our main results for arbitrary immanants of
complex characters -- and more generally, for their linear combinations
(i.e., class functions) -- we first establish the setting in which our
results hold.

\subsubsection{Ground ring -- contains character values}

The first step is to explain the degree of freedom in choosing the ground
ring. Fix an integer $n \geq 1$ and a unital commutative subring $R$.
Recall that given a complex character $\chi$ of the permutation group
$S_n$, the \textit{immanant} of a square matrix $A_{n \times n}$ -- as
defined by Littlewood and Richardson \cite{LR} -- is
\[
{\rm Imm}_\chi(A) := \sum_{\sigma \in S_n} \chi(\sigma)
\prod_{i=1}^n a_{i, \sigma(i)}.
\]
For this to act on $f(t u_i v_j)$ with $f \in R[[t]]$, we need $R$ to
``contain'' the character values of $\chi$. This is made precise as
follows:

\begin{definition}
Given a finite group $G$ and a complex character -- or class function --
$\psi := \sum_{\chi \in \widehat{G}_\mathbb{C}} a_\chi \chi$ (where the
sum runs over the irreducible complex characters of $G$), define the
ring $R_\psi \subset \mathbb{C}$ to be the unital subring generated by
all character values and all coefficients that occur in $\psi$:
\begin{equation}\label{Echarring}
R_\psi := \Z \left[ \{ \chi(g) : g \in G, \chi \in
\widehat{G}_\mathbb{C}, a_\chi \neq 0 \} \; \cup \; \{ a_\chi : \chi \in
\widehat{G}_\mathbb{C} \} \right] \subset \mathbb{C}.
\end{equation}
\end{definition}

In this paper we will work over arbitrary commutative $R_\psi$-algebras
$R$. For instance, if $\psi = \chi$ is the determinant or permanent, then
$R_\psi = \mathbb{Z}$, and so the determinant and permanent identities
will hold over every commutative $\Z$-algebra -- i.e.\ unital commutative
ring -- $R$, for all power series $f \in R[[t]]$. (Hence the theorems
above are stated over all $R$.)

\subsubsection{The Segre subalgebra -- contains the immanant}

The other point to establish is the algebra in which our immanant
identities are to hold. For this, we begin by reminding the reader that
we will work with arbitrary subgroups of the group of signed permutations
$S_2 \wr S_n$. (To understand this choice of group $S_2 \wr S_n$ and not
a larger one, see Section~\ref{Srootsof1}.) Now suppose $u_1, \dots, u_n,
v_1, \dots, v_n$ are commuting variables. Recall that $S_2 \wr S_n$ acts
on the tensor algebra $R \langle u_1, \dots, u_n \rangle$ via signed
permutations on the span of the $u_i$, and similarly on $R \langle v_1,
\dots, v_n \rangle$. Explicitly, write every element $g \in S_2 \wr S_n$
as: $g = D_g \sigma_g$, where $D_g = D_g^{-1}$ is a diagonal matrix with
$(i,i)$ diagonal entry $(D_g)_i \in \{ \pm 1 \}$, and $\sigma_g \in S_n$
is a permutation matrix. Now:
\begin{equation}\label{Eaction0}
g \cdot \sum_{i=1}^n r_i u_i := \sum_{i=1}^n r_i \cdot
(D_g)_{\sigma_g(i)} \cdot u_{\sigma_g(i)}
= \sum_{j=1}^n r_{\sigma_g^{-1}(j)} \cdot (D_g)_j \cdot u_j,
\end{equation}
and this is extended multiplicatively to $R \langle \bu \rangle$.
Moreover, this action preserves the two-sided ideals generated by
\[
\{ u_i \otimes u_j - u_j \otimes u_i : 1 \leq i,j \leq n \}, \qquad
\{ u_i \otimes u_j + u_j \otimes u_i : 1 \leq i < j \leq n \} \sqcup \{
u_i \otimes u_i : 1 \leq i \leq n \}.
\]

Hence denoting the free $R$-module $U := \sum_{i=1}^n R u_i$, the action
of $S_2 \wr S_n$ on $R \langle \bu \rangle$ descends to the quotient
symmetric algebra ${\rm Sym}_R^\bullet (U)$ and quotient alternating
algebra $\wedge_R^\bullet (U)$, via:
\begin{align}\label{Esymmalt}
\begin{aligned}
g(\bu^\bm) = &\ \prod_{i=1}^n ((D_g)_{\sigma_g(i)} \cdot
u_{\sigma_g(i)})^{m_i},\\
g( u_{i_1} \wedge \cdots \wedge u_{i_d} ) = &\ \prod_{j=1}^d
(D_g)_{\sigma_g(i_j)} \cdot ( u_{\sigma_g(i_1)} \wedge \cdots \wedge
u_{\sigma_g(i_d)} ),
\end{aligned}
\end{align}
for all non-negative integer tuples $\bm = (m_1, \dots, m_n)$ and all
integers $1 \leq i_1 < \cdots < i_d \leq n$.
Here and below, we define and use $\bu^\bm := \prod_{i=1}^n u_i^{m_i}$.
Notice that this action preserves each graded component
\begin{equation}\label{Egraded}
{\rm Sym}_R^d(\bu) := {\rm Sym}_R^d(U), \qquad
\wedge_R^d(\bu) := \wedge_R^d(U), \qquad d \geq 0.
\end{equation}

The above holds verbatim for the $v_j$, i.e.\ for the free $R$ module $V
:= \sum_{j=1}^n R v_j$. This action carries over to the tensor algebra $R
\langle \bv \rangle$ and its symmetric/alternating quotients.

Now fix a finite subgroup $G \leq S_2 \wr S_n$, an irreducible complex
character $\chi$ of $G$, and a commutative $R_\chi$-algebra $R$. Recall
that one has the ``minimal'' pseudo-idempotent in the group algebra
\[
E_\chi := \chi(1) \sum_{g \in G} \chi(g) g^{-1} \in
R_\chi G
\]
(see~\eqref{Echarring}), where ``pseudo-idempotent'' simply means that
$E_\chi^2 = |G| E_\chi$ in $R_\chi$ (and hence in $R_\psi$ for any class
function $\psi$ with $[\psi : \chi] \neq 0$). Below, we will act by
$E_\chi$ on both the $u_i$ and on the $v_j$, and so to keep track of
which variables are acted upon, we denote by $E_\chi^\bu$ the
pseudo-idempotent acting on $R \langle \bu \rangle$, and hence on each
$\bu^{\bf m}$ and each $\wedge_{j=1}^k u_{i_j}$ via~\eqref{Esymmalt}.
Similarly, $E_\chi^\bv$ will denote the pseudo-idempotent acting on $R
\langle \bv \rangle$ and hence on its quotients.

Our goal is to extend Theorem~\ref{Tsymm} to all characters of all finite
subgroups $G$ of $S_2 \wr S_n$. To do so, we will make these
pseudo-idempotents act on polynomials in both $u_i$ and $v_i$, i.e.\ on
the polynomial ring $R[\bu,\bv]$. More precisely, we work in its ``Segre
subring'' (see~\eqref{Egraded})
\begin{equation}\label{Esegre-sym}
\bigoplus_{d \geq 0} {\rm Sym}_R^d(\bu) \otimes {\rm Sym}_R^d(\bv)
\hookrightarrow {\rm Sym}^{2d}_R(U \oplus V).
\end{equation}
Note from above that $R \langle \bu, \bv \rangle = T^\bullet_R(U \oplus
V)$ is indeed a $(S_2 \wr S_n) \times (S_2 \wr S_n)$-submodule, as is its
quotient ${\rm Sym}^d_R(U \oplus V)$ for each $d \geq 0$ and hence the
Segre subring. In particular, the pseudo-idempotents $E_\chi^\bu$ act on
it as explained in~\eqref{Eaction} and by fixing all factors of $v_j$:
\[
E_\chi^\bu(\bu^\bm \bv^{\bm'}) := E_\chi^\bu(\bu^\bm) \cdot \bv^{\bm'} =
\chi(1) \sum_{g \in G} \chi(g) g^{-1}(\bu^\bm) \cdot \bv^{\bm'},
\]
and similarly for the $E_\chi^\bv$.

\subsection{The main theorem and its proof, for even variables}

Having defined the Segre subring and the action of the two
$E_\chi^\bullet$ on it, we can state the promised generalization of
Theorem~\ref{Tsymm} to all subgroups $G \leq S_2 \wr S_n$ and all
characters -- in fact, class functions -- $\psi$ of $G$. We begin with
$\psi$ a multiplicity-free character:

\begin{theorem}\label{Teven0}
Fix an integer $n \geq 1$, a subgroup $G \leq S_2 \wr S_n$ acting on the
variables $u_i$ (and on the $v_j$) by signed permutations, and a
multiplicity-free complex character $\chi$ of $G$. Then for any
commutative $R_\chi$-algebra $R$ -- see~\eqref{Echarring} -- and any
formal power series $f \in R[[t]]$ (with $t$ an indeterminate), one has:
\begin{equation}\label{Egeneral0}
|G| E_\chi^\bu \cdot \prod_{i=1}^n f(t u_i v_i) =
|G| E_{\overline{\chi}}^\bv \cdot \prod_{i=1}^n f(t u_i v_i) =
\sum_{{\bf m} \in \Z_{\geq 0}^n} t^{|{\bf m}|} f_{\bf m} \cdot
E_\chi^\bu(\bu^\bm) \cdot E_{\overline{\chi}}^\bv(\bv^\bm),
\end{equation}
where the indeterminate $t$ keeps track of the $\Z_{\geq 0}$-grading, we
use the multi-index notation
\[
\bm = (m_1, \dots, m_n), \quad |{\bf m}| = m_1 + \cdots + m_n, \quad
f_{\bf m} := \prod_i f_{m_i}, \quad
\bu^{\bf m} := \prod_i u_i^{m_i}, \quad
\bv^{\bf m} := \prod_i v_i^{m_i},
\]
and $\bm \geq {\bf 0}$ is interpreted coordinatewise.
\end{theorem}

Observe that special cases of Equation~\eqref{Egeneral0} yield Cauchy's
determinantal formula, its analogue for permanents and immanants (for the
power series $f_0(t) = 1/(1-t)$), and their generalizations to arbitrary
power series. E.g. for $\chi$ the sign and trivial representation
respectively (and $G = S_n$ for $n \geq 2$), the $G$-immanants have
``orthogonal'' expansions, respectively:
\begin{align}
n! \det f[ t \bu \bv^T ] = &\ n!\, V(\bu) V(\bv) \sum_{{\bf m} \in
\Z_{\geq 0}^n} t^{|{\bf m}| + \binom{n}{2}} \prod_{i=1}^n f_{m_i + n-i}
\cdot s_{\bf m}(\bu) s_{\bf m}(\bv),\label{Edet}\\
n! \perm f[ t \bu \bv^T ] = &\ \sum_{{\bf m} \in
\Z_{\geq 0}^n} t^{|{\bf m}|} f_{\bf m} \cdot \perm(\bu^{\circ {\bf m}})
\perm(\bv^{\circ {\bf m}}) \label{Eperm}\\
= &\ n! \sum_{{\bf m} \in \Z_{\geq 0}^n, \ {\bf m} \text{
non-increasing}} t^{|{\bf m}|} |{\rm Stab}_{S_n}({\bf m})| f_{\bf m}
\cdot m_{\bf m}(\bu) m_{\bf m}(\bv), \label{Eperm2}
\end{align}
for an arbitrary formal power series $f(t)$.
(Here $m_\bm(\bu)$ denotes the monomial symmetric polynomial.)
These equalities hold over an arbitrary unital commutative ring, and
hence one can work over $R = \mathbb{Q}[{\bf X}]$ a suitable polynomial
ring, cancel $n!$ from all of these, then observe the ``normalized''
identity over the subring $R = \Z[{\bf X}]$, and finally, specialize the
variables ${\bf X}$ to show these equalities over any unital commutative
ring $R$.

Before proving Theorem~\ref{Teven0}, we make two observations. The first
extends the theorem to all complex (finite-dimensional) characters of
$G$. Even more generally:

\begin{cor}
Setting as in Theorem~\ref{Teven0}. Let $\psi = \sum_{\chi \in
\widehat{G}_\mathbb{C}} a_\chi \chi$ be any complex class function of
$G$, and $R_\psi$ the corresponding ring as in~\eqref{Echarring}.
Defining
\[
E_\psi^\bu := \sum_{\chi \in \widehat{G}_\mathbb{C}} a_\chi E_\chi^\bu
\]
and similarly $E_{\overline{\psi}}^\bv$ (where we set $\overline{\psi} :=
\sum_\chi a_\chi \, \overline{\chi}$ -- using $a_\chi$ and not
$\overline{a}_\chi$), we have
\begin{equation}\label{Eclass}
|G| E_\psi^\bu \cdot \prod_{i=1}^n f(t u_i v_i) =
|G| E_{\overline{\psi}}^\bv \cdot \prod_{i=1}^n f(t u_i v_i) = \sum_{\chi
\in \widehat{G}_\mathbb{C}} a_\chi
\sum_{{\bf m} \in \Z_{\geq 0}^n} t^{|{\bf m}|} f_{\bf m} \cdot
E_\chi^\bu(\bu^\bm) \cdot E_{\overline{\chi}}^\bv(\bv^\bm)
\end{equation}
for an arbitrary $R_\psi$-algebra $R$ and any $f \in R[[t]]$ (and bosonic
variables $u_i, v_j$).
\end{cor}

While this identity is more general than~\eqref{Egeneral0}, it also is an
immediate consequence of it, by linearity. In fact, the proof (below) of
the first equality in~\eqref{Egeneral0} will also carry over verbatim
to~\eqref{Eclass}.

Our next remark explains why -- for $G = S_n$ or $S_2 \wr S_n$ -- the
above identities~\eqref{Egeneral0}, \eqref{Eclass} in fact hold over all
rings.

\begin{remark}\label{Ranyring}
Returning to~\eqref{Egeneral0}, two pleasing special cases are when $G =
S_n$ and $G = S_2 \wr S_n$ (the type $A$ and $B$ Weyl groups,
respectively). In this case, Springer showed \cite{Springer} that all
irreducible complex $G$-representations can in fact be constructed over
$\mathbb{Q}$. In particular, all character values are integers (since
they are algebraic integers and rational), and so $R_\chi = \Z$. Thus,
for $G$ the Weyl group of type $A$ or $B$, one has the immanant
identity~\eqref{Egeneral0} -- and hence the class function
identity~\eqref{Eclass} -- over arbitrary unital commutative rings.
\end{remark}

We now show the above theorem. The proof has two ingredients: the first
explains a key property of signed permutation matrices, when acting on
symmetric functions in two sets of variables.

\begin{lemma}\label{Lkeylemma}
Fix an integer $n \geq 1$ and a unital commutative ring $R$. Given a
signed permutation matrix $g = D \cdot \sigma \in GL_n(R)$ (see the lines
before~\eqref{Eaction0}), denote its action on the $u_i$
via~\eqref{Eaction0} by $g^\bu$, and similarly define $g^\bv$. These
extend to actions on the ring of polynomials $R[u_1, \dots, u_n, v_1,
\dots, v_n]$ and on its Segre subring. Then for every symmetric
polynomial $F$ in $n$ variables, and all signed permutations $g$, we have:
\[
g^\bu \cdot F(u_1 v_1, \dots, u_n v_n) = (g^{-1})^\bv \cdot F(u_1 v_1,
\dots, u_n v_n), \qquad \forall F \in R[w_1, \dots, w_n]^{S_n}, \ g \in
S_2 \wr S_n.
\]
\end{lemma}

\noindent We defer the proof as we will also show the converse result, in
Theorem~\ref{Tsignedperm} below.

With Lemma~\ref{Lkeylemma} in hand, we can complete the proof of the
theorem above.

\begin{proof}[Proof of Theorem~\ref{Teven0}]
We begin with an arbitrary power series $f(t) = \sum_{m \geq 0} f_m t^m
\in R[[t]]$, and assert the equation
\begin{equation}\label{Etrivial0}
\prod_{i=1}^n f(t u_i v_i) = \sum_{{\bf m} \in \Z_{\geq 0}^n}
t^{|{\bf m}|} f_{\bf m} \bu^{\bf m} \bv^{\bf m}.
\end{equation}

Notice that Equation~\eqref{Etrivial0} is (a) obvious, and (b) precisely
the sought-for identity (Equation~\eqref{Egeneral0}) corresponding to the
trivial group $G = \{ 1 \}$.

We now return to the original setting of a general subgroup $G \leq S_2
\wr S_n$ acting on the $u_i$ and on the $v_j$ via signed
permutations~\eqref{Eaction0} -- and a multiplicity-free (complex)
character $\chi$ of $G$. Working in the Segre subring~\eqref{Esegre-sym},
apply the operators $E_\chi^\bu$ and $E_{\overline{\chi}}^\bv$ to the
above equation~\eqref{Etrivial0}.\footnote{If $G = S_n$, then this
precisely yields the corresponding immanant of the matrix $f[t \bu \bv^T
]$.}
We now claim that both operations yield equal expressions on the
left-hand side by reindexing. Indeed, apply Lemma~\ref{Lkeylemma} with
\[
F(u_1 v_1, \dots, u_n v_n) = \prod_{i=1}^n f(t u_i v_i) \in R[u_1,
\dots, u_n, v_1, \dots, v_n][[t]].
\]
This yields the following calculation -- e.g.\ in each $t$-degree
separately:
\begin{align*}
E_{\overline{\chi}}^\bv \cdot \prod_{i=1}^n f(t u_i v_i) = &\
\overline{\chi}(1) \sum_{h = g^{-1} \in G} \overline{\chi}(h^{-1})
(g^{-1})^\bv \cdot F(u_1 v_1, \dots, u_n v_n)\\
= &\ \chi(1) \sum_{g \in G} \overline{\chi(g)} g^\bu \cdot F(u_1 v_1,
\dots, u_n v_n) = E_\chi^\bu \cdot \prod_{i=1}^n f(t u_i v_i),
\end{align*}
since $\overline{\chi(g)} = \chi(g^{-1})$ for all $g \in G$.

This implies that both operations yield the same result \textit{on the
right-hand side of~\eqref{Etrivial0} as well}.
In particular, $E_\chi^\bu \cdot E_{\overline{\chi'}}^\bv = 0$ when
acting on~\eqref{Etrivial0}, for irreducible complex characters $\chi
\neq \chi'$, since
\[
E_\chi^\bu \cdot E_{\overline{\chi'}}^\bv \cdot \sum_{{\bf m} \in
\Z_{\geq 0}^n} t^{|{\bf m}|} f_{\bf m} \bu^{\bf m} \bv^{\bf m} =
E_\chi^\bu \cdot E_{\overline{\chi'}}^\bv \cdot \prod_{i=1}^n f(t u_i
v_i) = E_\chi^\bu E_{\chi'}^\bu \cdot \prod_{i=1}^n f(t u_i v_i),
\]
and this vanishes by character orthogonality. This implies that
$E_\chi^\bullet$ is also pseudo-idempotent for $\chi$ multiplicity-free
($E_\chi^2 = |G| E_\chi$),
and so applying either $|G| E_\chi^\bu$ or $|G| E_{\overline{\chi}}^\bv$
to the left-hand side of~\eqref{Etrivial0} is the same as applying
$E_\chi^\bu \cdot E_{\overline{\chi}}^\bv$:
\[
|G| E_\chi^\bu \cdot \prod_{i=1}^n f(t u_i v_i) = (E_\chi^\bu)^2 \cdot
\prod_{i=1}^n f(t u_i v_i) = E_\chi^\bu \cdot E_{\overline{\chi}}^\bv
\prod_{i=1}^n f(t u_i v_i).
\]
Therefore, the same observation applies to the right-hand side
of~\eqref{Etrivial0} -- which yields the result.
\end{proof}

\subsection{Larger linear groups do not work}\label{Srootsof1}

We now explain -- as promised above -- why Theorem~\ref{Teven0} does not
extend to other finite subgroups $G \leq GL_n$. The proof of
Theorem~\ref{Teven0} was in three steps:
\begin{enumerate}
\item Lemma~\ref{Lkeylemma}, which says that the actions of
\[
g^\bu,\ (g^{-1})^\bv : R[w_1, \dots, w_n]^{S_n} \to \bigoplus_{d \geq 0}
{\rm Sym}^d_R(\bu) \otimes {\rm Sym}^d_R(\bv)
\]
are the same, where $w_i = u_i v_i$ -- \textbf{if} $g$ is a signed
permutation.

\item This implies that the actions of $E_\chi^\bu$ and
$E_{\overline{\chi}}^\bv$ are the same on the symmetric function
$\prod_{i=1}^n f(t u_i v_i)$, for any multiplicity-free character $\chi$
of any $G \leq S_2 \wr S_n$.

\item Now the pseudo-idempotence of $E_\chi^\bu$ and
$E_{\overline{\chi}}^\bv$ implies the result.
\end{enumerate}

Given Theorem~\ref{Teven0}, it is now natural to ask if this result can
be extended from $G \leq S_2 \wr S_n$ to any finite matrix subgroup $G$
of $GL_n(R)$. In greater detail: first note that the action of $S_n$ (or
$S_2 \wr S_n$) on the free $R$-module $U$ extends to that of matrices $g
= (m_{ij})_{i,j=1}^n \in GL_n$ via:
\begin{equation}\label{Eaction}
g \cdot \sum_{j=1}^n r_j u_j = \sum_{i=1}^n \left( \sum_{j=1}^n m_{ij}
r_j \right) u_i.
\end{equation}
In turn, this $GL_n$-action extends to all of $R \langle \bu \rangle$ by
multiplicativity, and then descends to a $GL_n(R)$-action on the quotient
algebras ${\rm Sym}^\bullet_R(\bu), \wedge^\bullet_R(\bu)$. These remarks
apply equally to $V$ and $U \oplus V$ in place of $U$, and then one can
ask if Theorem~\ref{Tsymm} extends to characters of a finite subgroup $G
\leq GL_n(R)$ that need not be contained in $S_2 \wr S_n$.

Here we show two negative results. The first is a converse to
Step~(1), and shows that over an integral domain, the conclusion of
Lemma~\ref{Lkeylemma} holds only for signed permutations:

\begin{theorem}\label{Tsignedperm}
Fix a unital commutative ring $R$, an integer $n \geq 1$, and bosonic
indeterminates $u_i, v_i$ for $i=1,\dots,n$. Given an element $g \in
GL_n(R)$, each of the following assertions implies the next:
\begin{enumerate}
\item $g$ is a signed permutation: $g \in S_2 \wr S_n$.

\item $g^\bu \cdot F(u_1 v_1, \dots, u_n v_n) = 
(g^{-1})^\bv \cdot F(u_1 v_1, \dots, u_n v_n)$
for all symmetric functions $F \in R[w_1, \dots, w_n]^{S_n}$.

\item $g^\bu \cdot F(u_1 v_1, \dots, u_n v_n) = 
(g^{-1})^\bv \cdot F(u_1 v_1, \dots, u_n v_n)$
for the $n$ elementary symmetric functions
\[
e_k({\bf w}) = \sum_{1 \leq i_1 < i_2 < \dots < i_k \leq n} w_{i_1}
w_{i_2} \cdots w_{i_k}, \qquad 1 \leq k  \leq n.
\]
\end{enumerate}

\noindent If moreover $R$ is an integral domain, then all assertions are
equivalent.
\end{theorem}

\begin{remark}
The group $S_2 \wr S_n$ of signed permutations affords several attractive
properties over the reals $R = \R$, i.e.\ as a subgroup of $GL_n(\R)$. In
addition to its irreducible representations being constructible over
$\mathbb{Q}$ (being the type $B$ Weyl group; see Remark~\ref{Ranyring}),
signed permutation matrices enjoy characterizations in multiple fields.
In linear algebra, they are precisely the orthogonal matrices with
integer entries. In analysis, as a special case of the Banach--Lamperti
theorem, they coincide with the linear isometries of the $p$-norms
$(\mathbb{R}^n, \| \cdot \|_p)$ for each $p \in [1,\infty] \setminus \{ 2
\}$. Now our Theorem~\ref{Tsignedperm} provides a ``symmetric function''
characterization in $GL_n(R)$ of the signed permutation matrices -- over
any integral domain $R$ -- that is novel to the best of our knowledge.
\end{remark}

\begin{proof}[Proof of Theorem~\ref{Tsignedperm}]
We show a cyclic chain of implications, starting with $(1) \implies (2)$
(which was Lemma~\ref{Lkeylemma}). Write $g = D \sigma$, where $D$ is a
diagonal matrix with $(i,i)$ entry $\varepsilon_i \in \{ \pm 1 \}$. Now,
\[
g^\bu \cdot F(u_1 v_1, \dots, u_n v_n) = D^\bu \sigma^\bu \cdot F(u_1
v_1, \dots, u_n v_n) = F \left( ( \varepsilon_{\sigma(i)} u_{\sigma(i)}
v_i)_{i=1}^n \right),
\]
whereas
\[
(g^{-1})^\bv \cdot F(u_1 v_1, \dots, u_n v_n) = (\sigma^{-1})^\bv
(D^{-1})^\bv \cdot F \left( (u_j v_j)_{j=1}^n \right) = F \left( (
u_j \varepsilon_j^{-1} v_{\sigma^{-1}(j)} )_{j=1}^n \right).
\]
Now permute the arguments here via: $j = \sigma(i)$, and use that
$\varepsilon_j^{-1} = \pm 1 = \varepsilon_j$ together with the symmetry
of $F$, to conclude the proof.

Clearly, $(2) \implies (3)$. Now we suppose $R$ is an integral domain,
say with quotient field $\F$, and show that $(3) \implies (1)$.
We begin by recalling an observation on elementary symmetric functions
that is required in this proof. Suppose an infinite field $\mathbb{K}$
contains pairwise distinct elements $w_1, w_2, \dots, w_n$ and pairwise
distinct elements $w'_1, w'_2, \dots, w'_n$, whose elementary symmetric
functions agree:
\[
e_1({\bf w}) = w_1 + \cdots + w_n = w'_1 + \cdots + w'_n = e_1({\bf w}'),
\quad e_2({\bf w}) = e_2({\bf w}'), \quad \dots, \quad e_n({\bf w}) =
e_n({\bf w}').
\]
Then the polynomials $(x-w_1) \cdots (x-w_n)$ and $(x-w'_1) \cdots
(x-w'_n)$ coincide in $\mathbb{K}[x]$, hence so do their sets of roots in
the field $\mathbb{K}$ -- i.e., $\{ w_i : 1 \leq i \leq n \} = \{ w'_i :
1 \leq i \leq n \}$. We will apply this observation presently, with the
(infinite) field being $\mathbb{K}' := \F(u_1, \dots, u_n, v_1, \dots,
v_n)$.

Returning to the proof, let $g = (m_{ij})_{i,j=1}^n \in GL_n(R)$, and
denote $\varepsilon := \det(g) \in R^\times$. Also write the adjugate
matrix of $g$ as ${\rm adj}(g) = (a_{ij})_{i,j=1}^n$, where $a_{ij}$
equals $(-1)^{i+j}$ times the $(j,i)$-minor of $g$. In particular,
$g^{-1} = (\varepsilon^{-1} a_{ij})_{i,j=1}^n$.
Now compute, for $F$ running over the elementary symmetric polynomials in
$n$ variables:
\begin{align*}
g^\bu \cdot F(u_1 v_1, \dots, u_n v_n) = &\ F \left( \left( v_i
\textstyle{\sum_{j=1}^n} m_{ji} u_j \right)_{i=1}^n \right),\\
(g^{-1})^\bv \cdot F(u_1 v_1, \dots, u_n v_n) = &\ F \left( \left( u_i
\varepsilon^{-1} \textstyle{\sum_{j=1}^n} a_{ji} v_j \right)_{i=1}^n \right).
\end{align*}

We now apply the above observation applied to $\mathbb{K}'$; notice this
is possible because as $g = (m_{ij})$ is invertible, no row or column is
zero, and so the $i$th argument of $g^\bu \cdot F$ is a nonzero multiple
of $v_i$ in $\F[u_1, \dots, u_n, v_1, \dots, v_n]$, but not of any other
$v_j$. Similarly for the arguments of $(g^{-1})^\bv \cdot F$. Hence by
the above observation, there exists a permutation $\sigma \in S_n$ such
that
\[
v_{\sigma(i)} \sum_{j=1}^n m_{j \sigma(i)} u_j = u_i \varepsilon^{-1}
\sum_{j=1}^n a_{ji} v_j, \qquad \forall 1 \leq i \leq n.
\]

But this is possible in the rational function field $\mathbb{K}'$ only if
the coefficients of every $u_r v_s$ are equal on both sides. Thus,
$m_{j\sigma(i)} = 0$ whenever $j \neq i$ (so $g$ is necessarily a
``generalized permutation matrix''). Moreover, equating the coefficients
of $u_i v_{\sigma(i)}$ on both sides yields:
\begin{equation}\label{Etemp}
m_{i \sigma(i)} = \varepsilon^{-1} a_{\sigma(i) i} = \det(g)^{-1}
a_{\sigma(i) i}.
\end{equation}

Now expand the determinant of $g$ along its $\sigma(i)$-column -- where
we saw above that only the $i$th entry is nonzero. Thus,
\[
\varepsilon = \det(g) = \sum_{j=1}^n m_{j \sigma(i)} a_{\sigma(i) j} =
m_{i \sigma(i)} a_{\sigma(i) i} \in R^\times.
\]
Substituting this in~\eqref{Etemp} finally gives us:
\[
m_{i \sigma(i)} = \det(g)^{-1} a_{\sigma(i) i} = m_{i \sigma(i)}^{-1}
\quad \implies \quad m_{i \sigma(i)} = \pm 1.
\]

Thus, $g = (m_{ij})$ is a matrix with exactly one nonzero entry in each
row and column, and each such entry is $\pm 1$. Hence $g \in S_2 \wr
S_n$, which shows $(3) \implies (1)$.
\end{proof}

Theorem~\ref{Tsignedperm} shows that over an integral domain, Step~(1) at
the start of this discussion (i.e.\ of Section~\ref{Srootsof1}) fails to
hold for any matrix $g$ that is not a signed permutation.
In order to check if the proof of Theorem~\ref{Teven0} still goes through
for larger matrix groups $G$, the next strategy to attempt would be to
directly show Step~(2) without showing Step~(1). Namely, we would compute
-- and equate -- the actions of $|G| E_\chi^\bu, |G|
E_{\overline{\chi}}^\bv$ by summing over the entire group $G$, instead of
using individual terms $g^\bu, (g^{-1})^\bv$ and reindexing via $h =
g^{-1}$.

Unfortunately, this strategy also fails -- already in ``small'' cases:

\begin{example}\label{Excuberoot}
As a simple (counter)example, we now show that~\eqref{Egeneral0} fails
for $n=1$ and
\[
G = \{ \exp(2 \pi k i / 3) : k = 0,1,2 \} \cong
\Z / 3 \Z \hookrightarrow S^1 \leq \mathbb{C}^\times,
\]
for a certain character. (Here, we work over $R = \mathbb{C}$.) More
generally, let $d \geq 3$ and
\[
G = \{ \exp(2 \pi k i / d) : k = 0, 1, \dots, d-1 \} \cong \Z / d \Z
\hookrightarrow S^1 \leq \mathbb{C}^\times.
\]
Let the character $\chi(g) := g \in G$, and write $\zeta := \exp(2 \pi i
/ d)$. Now $E_\chi^{u_1}$ acts on $f(t u_1 v_1) = \sum_{m \geq 0} f_m t^m
u_1^m v_1^m$ via:
\begin{align*}
E_\chi^{u_1} \cdot f(t u_1 v_1) = &\ \sum_{g \in G} \sum_{m=0}^\infty f_m
t^m v_1^m \cdot \chi(g^{-1}) (g u_1)^m\\
= &\ \sum_{m=0}^\infty t^m \cdot f_m u_1^m v_1^m \cdot \sum_{a \in \Z / d
\Z} \zeta^{a(m-1)}.
\end{align*}

Similarly,
\[
E_{\overline{\chi}}^{v_1} \cdot f(t u_1 v_1) = \sum_{m=0}^\infty t^m
\cdot f_m u_1^m v_1^m \cdot \sum_{a \in \Z / d \Z} \zeta^{a(m+1)},
\]
and for $d \geq 3$, this series differs from the preceding (formal) power
series in $(2/d)$ths of the coefficients: for all $m \equiv \pm 1 \mod
d$. \qed
\end{example}

Theorem~\ref{Tsignedperm} and Example~\ref{Excuberoot} explain our choice
of working with $G \leq S_2 \wr S_n$ in formulating and proving the main
results: larger groups do not work even for small $n$.

\subsection{Further characterizations of the signed permutation matrices}

For completeness, here we provide some explanation for why
Lemma~\ref{Lkeylemma} and Theorem~\ref{Tsignedperm} characterize the
group of signed permutations. Suppose $R = \mathbb{F}$ is an infinite
field; we will examine the action of the group $GL_n(\F)$ on the vector
space $\mathcal{M}^{n \times n}$, where $\mathcal{M} \cong U \otimes V$
is the $\F$-span of the $n^2$ vectors $u_i v_j$.

The first point is that any $g \in GL_n(\F)$ acts on all $u_i \in U$ via
$g^\bu$, which is simply matrix multiplication on $\bu$, and hence on
$\bu \bv^T$. The action of $g^\bv$ on $V$ is similar -- but in this
instance it is from the right, hence via its transpose; thus to make it a
valid group action, we need to apply yet another anti-involution. It
follows that the two actions $g \mapsto g^\bu$ and $g \mapsto
((g^{-1})^T)^\bv$ are indeed group actions, when thought of as acting on
$\mathcal{M}^{n \times n}$ from the left and the right, respectively.
Thus, Proposition~\ref{Pcharacterize} below serves as an explanation for
why signed permutations are related to $g^\bu$ acting on $\bu$ and
$((g^{-1})^T)^\bv$ acting on $\bv^T$, as in Theorem~\ref{Tsignedperm}.

Second, via the observation in the proof of $(3) \implies (1)$ of
Theorem~\ref{Tsignedperm}, we are now interested classifying the $g \in
GL_n(\F)$ for which the diagonal entries in $g^\bu (\bu \bv^T)$ and in
$(\bu \bv^T) ((g^{-1})^T)^\bv$ coincide (up to permutation). Such $g$ are
characterized by Theorem~\ref{Tsignedperm} to be precisely the signed
permutation matrices. As we now explain, the connection to symmetric
function theory is essentially ``solely'' via the aforementioned
observation. In other words, modulo this fact one obtains a ``linear
algebra'' characterization of $S_2 \wr S_n \leq GL_n(\F)$.\medskip

To explain, begin by fixing $g \in GL_n(\mathbb{F})$. Instead of asking
that $g^\bu$ and $((g^{-1})^T)^\bv$ fix the set of diagonal entries of
all matrices in $\mathcal{M}^{n \times n}$, we work alternately with rank
one matrices $u v^T \in \F^{n \times n}$ instead of $\bu \bv^T \in
\mathcal{M}^{n \times n}$. (Note that since the ground field is infinite,
(symmetric) polynomials are the same as (symmetric) polynomial
functions.) Since $g$ is fixed, by relabelling $v$ to $v_g := gv$ one can
ask to classify those $g$ such that the diagonal entries of
\[
g \cdot (u(gv)^T) = g \cdot u v^T \cdot g^T \qquad \text{and of} \qquad
(u(gv)^T) \cdot (g^{-1})^T = uv^T
\]
agree as (multi)sets, for all vectors $u,v \in \F^n$. By linearity, this
would imply the same fact with $u v^T$ replaced by any matrix $A \in
\F^{n \times n}$. Characterizing such $g$ is the content of our next
result, shown not just over $\mathbb{F}^{n \times n}$, but again in the
generality of Theorem~\ref{Tsignedperm}:

\begin{proposition}\label{Pcharacterize}
Fix a unital commutative ring $R$ and an integer $n \geq 1$. Given an
element $g \in GL_n(R)$, each of the following assertions implies the
next:
\begin{enumerate}
\item $g$ is a signed permutation: $g \in S_2 \wr S_n$.

\item For all $A \in R^{n \times n}$, $g A g^T$ and $A$ have the same
multisets of diagonal entries.

\item For all $1 \leq i \leq n$, the matrices $g E_{ii} g^T$ and $E_{ii}$
have the same multisets of diagonal entries, where $E_{ii}$ is the matrix
with $(i,i)$ entry $1$ and all other entries $0$.
\end{enumerate}

\noindent If moreover $R$ is an integral domain, then all assertions are
equivalent.
\end{proposition}

Note that one can replace any ``intermediate'' subset
\[
\{ E_{11}, \dots, E_{nn} \} \, \subset \, S \, \subset \, R^{n \times n}
\]
in the above characterization of $S_2 \wr S_n \leq GL_n(R)$ for $R$ an
integral domain. (Also, as the following proof shows, we require from the
ground ring $R$ not that it is an integral domain, but only that the
square roots in $R$ of $0$ and $1$ are $\{ 0 \}$ and $\{ \pm 1 \}$,
respectively.)

\begin{proof}
We show a cyclic chain of implications. If $(1)$ holds, write $g = D
\sigma$; then $g^{-1} = g^T = \sigma^T D$. Since $D$ has diagonal entries
$\pm1$, it is easy to check that $D A D$ and $A$ have the same multisets
of diagonal entries. So do $\sigma A \sigma^T$ and $A$, and so~$(2)$
follows.

Clearly, $(2) \implies (3)$. We now assume $(3)$, as well as that $R$ is
an integral domain, and show $(1)$ -- in two ways. The first is by direct
computation: write the columns of $g \in GL_n(R)$ as: $g = [v_1 | \cdots
| v_n]$, with each $v_j$ having entries $g_{ij} \in R$. Now if $D$ is a
diagonal matrix with $(i,i)$ entry $d_i \in R$, then a direct computation
reveals:
\begin{equation}\label{Ecalculation}
g D g^T = \sum_{j=1}^n d_j v_j v_j^T.
\end{equation}
In particular, $g E_{jj} g^T = v_j v_j^T$, whose diagonal entries are $\{
g_{ij}^2 : 1 \leq i \leq n \}$. By hypothesis, this multiset equals $\{
1, 0, \dots, 0 \}$, so there exists a map $\sigma : \{ 1, \dots, n \} \to
\{ 1, \dots, n \}$ such that $v_j$ equals the standard basis vector
$e_{\sigma(j)}$ or else $-e_{\sigma(j)}$. Since $g$ is invertible,
$\sigma$ has to be injective, hence $\sigma \in S_n$, and so $g \in S_2
\wr S_n$ as desired.

The second proof of $(3) \implies (1)$ is slightly different, and starts
by considering the diagonal entries in~\eqref{Ecalculation}: the $(i,i)$
entry is:
\[
(g D g^T)_{ii} = \sum_{j=1}^n d_j g_{ij}^2.
\]
For each $D = E_{i'i'}$, this yields a system of $n$ equations, which can
be stated using the coefficient matrix $(g_{ij}^2)$ applied to the
standard coordinate basis vector $e_{i'} \in R^n$. Collecting these
systems together for all $i'$ yields
\[
\begin{pmatrix}
g_{11}^2 & g_{12}^2 & \cdots & g_{1n}^2 \\
g_{21}^2 & g_{22}^2 & \cdots & g_{2n}^2 \\
\vdots & \vdots & \ddots & \vdots \\
g_{n1}^2 & g_{n2}^2 & \cdots & g_{nn}^2
\end{pmatrix} 
\cdot [ \, e_1 \, | \, e_2 \, | \, \cdots \, | \, e_n \, ] = 
[ \, e_{\sigma(1)} \, | \, e_{\sigma(2)} \, | \, \cdots \, | \,
e_{\sigma(n)} \, ]
\]
for some map $\sigma : \{ 1, \dots, n \} \to \{ 1, \dots, n \}$. Thus
for each $j$, $g_{ij}^2 = 0$ for $i \neq \sigma(j)$, and
$g_{\sigma(j)j}^2 = 1$ in $R$. Hence $g_{\sigma(j)j} \in \{ \pm 1 \}$ for
each $j$. Now if $\sigma$ is not a bijection then $\sigma(j) =
\sigma(j')$ for some $1 \leq j \neq j' \leq n$, in which case the $j,j'$
columns of $g \in GL_n(R)$ are proportional. This is impossible, so
$\sigma$ is a bijection and each $g_{\sigma(j)j} = \pm 1$. Thus $g \in
S_2 \wr S_n$, as desired.
\end{proof}

\section{Immanant identities for fermionic variables}

Theorem~\ref{Teven0} holds in the case of even/bosonic variables, i.e.,
where the $u_i, v_j$ all commute among themselves. Our next result is an
``odd''/fermionic analogue of Theorem~\ref{Teven0}, in which the $u_i$ and
$v_j$ pairwise anti-commute: $u_i u_j = - u_j u_i$, and similarly for
$v_i, v_j$ and for $u_i, v_j$. (In particular, we also require $u_j^2 =
v_j^2 = 0$ for $j \geq 1$.) We continue to work with a subgroup $G \leq
S_2 \wr S_n$, and with a power series over some $R_\psi$-algebra when
dealing with a complex class function $\psi$ of $G$. We also fix an
ordering of the fermionic variables $u_i$ and the same one for the $v_i$,
say increasing indices $(1, 2, \dots, n)$.

Since $u_j^2 = v_j^2 = 0\ \forall j \geq 1$, without loss of generality
$f(t) = f_0 + f_1 t$ is linear, and so the fermionic analogue of
Equation~\eqref{Etrivial0} is
\begin{equation}\label{Eodd}
\prod_{j=1}^n f(t u_j v_j) = \sum_{J \subset [n]} (-1)^{\binom{|J|}{2}}
f_0^{n-|J|} (f_1 t)^{|J|} \bu^J \bv^J,
\end{equation}

\noindent where $[n] := \{ 1, \dots, n \}$,
the power of $(-1)$ emerges from ``taking the $u_i$ past the $v_j$'', and
we use the notation
\[
\bu^J = \prod_{j \in J} u_j, \quad \bv^J = \prod_{j \in J} v_j.
\]

As in the case of even variables (and forgetting the role of $t$),
Equation~\eqref{Eodd} takes place inside the alternating algebra, or more
precisely, inside the $G \times G$-submodule
\[
\bigoplus_{d \geq 0} \wedge^d_R(\bu) 
\otimes \wedge^d_R(\bv) \subset \wedge{}_R^\bullet(M^1)
\]
(see the discussion around~\eqref{Esymmalt}), where $M^1$ is the free
$R$-module with basis $\{ u_1, \dots, u_n, \ v_1, \dots, v_n \}$. (Here
one begins with the aforementioned $(S_2 \wr S_n) \times (S_2 \wr
S_n)$-module structure on $T_\mathbb{C}^\bullet(M^1)$.)
In particular, the two pseudo-idempotents $E_\chi^\bu, E_\chi^\bv$ act on
this alternating algebra. Now we claim that
applying $E_\chi^\bu$ or $E_{\overline{\chi}}^\bv$ to the left-hand side
of Equation~\eqref{Eodd} yields the same expression.
To see why, note that the proof of Lemma~\ref{Lkeylemma} goes through for
fermionic variables as well, the key being that $u_i \cdot v_j$ is even
for all $i,j$, hence polynomials in these are well-defined.

But this implies the same result on the right-hand sides of~\eqref{Eodd}
-- as in the bosonic case. Since $E_\chi^\bu, E_{\overline{\chi}}^\bv$
are pseudo-idempotents, this implies the sought-for ``fermionic''
immanant identity for irreducible characters -- which we extend by
linearity as in~\eqref{Eclass}:

\begin{theorem}\label{Todd}
Fix an integer $n \geq 1$, a subgroup $G \leq S_2 \wr S_n$ (acting on
$u_i, v_j$ by signed permutations), and a complex class function $\psi =
\sum_{\chi \in \widehat{G}_\mathbb{C}} a_\chi \chi$ of $G$.
Let $R_\psi$ be the corresponding ring as in~\eqref{Echarring}, and $R$
an $R_\psi$-algebra.
Working with fermionic variables $u_i, v_j$, and for $f \in R[[t]]$ an
arbitrary formal power series with $t$ an indeterminate, one has:
\begin{align}\label{Egeneral-odd}
\begin{aligned}
|G| E_\psi^\bu \cdot \prod_{j=1}^n f(t u_j v_j) = &\
|G| E_{\overline{\psi}}^\bv \cdot \prod_{j=1}^n f(t u_j v_j)\\
= &\ \sum_{\chi \in \widehat{G}_\mathbb{C}} a_\chi
\sum_{J \subset [n]} (-1)^{\binom{|J|}{2}} f_0^{n-|J|} (f_1 t)^{|J|}
E_\chi^\bu (\bu^J) \cdot E_{\overline{\chi}}^\bv (\bv^J).
\end{aligned}
\end{align}
\end{theorem}

As in the bosonic case, a prominent special case is that of $G = S_n$,
with $\chi$ the sign and trivial representations. By
Remark~\ref{Ranyring}, we may work over any commutative ring $R$.
Since $x_{ij} := u_i v_j$ is still an even variable for all $1 \leq i,j
\leq n$, the $x_{ij}$ commute pairwise and so one can expand the
determinant (and permanent) along any row or column. The expansions turn
out to be ``mirror images'' that involve only the two largest or the two
smallest powers of $t$:

\begin{proposition}\label{Podd}
Fix an integer $n \geq 2$, a unital commutative ring $R$, and $f_0, f_1
\in R$. Given odd variables $u_i, v_j$ for $1 \leq i,j \leq n$ as above,
\begin{align}\label{Edetodd}
&\ \det (f_0 + f_1 t u_i v_j)_{i,j=1}^n \notag\\
= &\ t^n (-1)^{\binom{n}{2}} n!
f_1^n \cdot u_1 \cdots u_n \cdot v_1 \cdots v_n\\
& + t^{n-1} (-1)^{\binom{n-1}{2}} (n-1)! f_0 f_1^{n-1} \cdot
\sum_{i=1}^n (-1)^{i-1} u_1 \cdots \widehat{u_i} \cdots u_n \cdot
\sum_{j=1}^n (-1)^{j-1} v_1 \cdots \widehat{v_j} \cdots v_n.\notag
\end{align}
Similarly, the permanent of the above matrix equals
\begin{equation}\label{Epermodd}
\perm (f_0 + f_1 t u_i v_j)_{i,j=1}^n = n! f_0^n + (n-1)!
f_0^{n-1} f_1 t (u_1 + \cdots + u_n)(v_1 + \cdots + v_n).
\end{equation}
\end{proposition}

\begin{remark}
Notice that this is not immediately connected to the even-variable case,
since if one specializes the equation in~\eqref{Edet} to $G = S_n$,
$\chi$ the sign representation, and $f(t) = f_0 + f_1 t$, then already
for $n \geq 3$ all products in~\eqref{Edet} vanish, so we simply get zero
there.
\end{remark}

\begin{proof}[Proof of Proposition~\ref{Podd}]
While one can verify the claimed identities via explicit computations, we
show them as corollaries of the ``master immanant
identity''~\eqref{Egeneral-odd}, by computing the $J$-summand for every
$J \subset [n]$. Since we work over any commutative $R$, we may assume
without loss of generality that $R$ has characteristic zero (see the
discussion after~\eqref{Eperm2}).

We begin with the determinant~\eqref{Edetodd}, i.e.\ with $\chi$ the sign
character. Here, the term $E_\chi^\bu(\bu^J)$ equals a signed sum
over permutations (the Laplace expansion), and one has three cases:
\begin{itemize}
\item When $J = [n]$, $E_\chi^\bu(\bu^J)$ yields $n!$ copies
of $(f_1 t)^n$ times the \textit{ordered} product $u_1 \cdots u_n$, and
similarly for the $v_j$. Now dividing by $n!$ (since the left-hand side
of~\eqref{Egeneral-odd} equals $n! \det (f(t u_i v_j))$) yields the first
term on the right in~\eqref{Edetodd}.

\item For each $J \subset [n]$ with $|J| = n-1$, a similar analysis shows
that the $n!$ terms in $E_\chi^\bu(\bu^J)$ yield (with equal
multiplicities) the ordered products $(-1)^{i-1} u_1 \cdots \widehat{u_i}
\cdots u_n$, each with multiplicity $(n-1)!$. This also happens for
$E_\chi^\bv(\bv^J)$; and there are $\binom{n}{n-1}$ choices of such
subsets $J$. Thus the overall multiplicity for the term
in~\eqref{Edetodd} of $t$-degree $|J|=n-1$ is $n \cdot (n-1)!^2$. Now
divide by $n!$ as above; this yields $(n-1)!$, and hence the second term
on the right in~\eqref{Edetodd}.

\item Finally, fix $J \subset \{ 1, \dots, n \}$ with $|J| \leq n-2$.
Also fix $i < i'$ in $[n] \setminus J$. Then $E_\chi^\bu(\bu^J)$ is a
signed sum over $\sigma \in S_n$, which can be rewritten over the cosets
of the reflection subgroup $\{ 1, (i,i') \}$. As multiplying by the
transposition $(i,i')$ reverses sign, this signed sum vanishes for every
$J$ of size at most $n-2$.
\end{itemize}

This shows the determinant formula~\eqref{Edetodd}. We now turn to the
permanent~\eqref{Epermodd}; the proof here is similar and we write it out
for completeness. Once again, there are three cases for $J \subset [n]$:
\begin{itemize}
\item When $J = \emptyset$, $E_\chi^\bu(\bu^J) = E_\chi^\bu(1) = n!$, and
similarly for $E_\chi^\bv(\bv^J)$. Now divide by 
$n!$ (since the left-hand side of~\eqref{Egeneral-odd} equals $n!$ times
the permanent) to obtain the $n! f_0^n$ term in the claimed expansion.

\item When $J = \{ j \}$ is a singleton, a similar computation shows that
$E_\chi^\bu(u_j)$ yields $(n-1)!$ copies of each $u_i$. Add all of these
terms, then do the same for the action of $E_\chi^\bv$ on $v_j$
(throughout, $J = \{ j \}$ is fixed). Now summing over all $j$, and
dividing by $n!$ as above, we get $(n-1)! f_0^{n-1} f_1 t (u_1 + \cdots +
u_n)(v_1 + \cdots + v_n)$.

\item Finally, if $|J| \geq 2$, we choose $j<j'$ in $J$, and write 
$E_\chi^\bu(\bu^J)$ as the (unsigned) sum over $\sigma \in S_n$, which
can be rewritten over the cosets of the reflection subgroup $\{ 1, (j,j')
\}$. As multiplying by the transposition $(j,j')$ reverses sign, this
unsigned average vanishes for every $J$ of size at least $2$. \qedhere
\end{itemize}
\end{proof}

For completeness, we conclude this part with a fermionic counterpart of
two related results for bosonic variables -- Cauchy's \textit{product}
identities:
\begin{equation}\label{Eprodeven}
\prod_{i,j} \frac{1}{1 - u_i v_j} = \sum_\bm s_\bm(\bu) s_\bm(\bv),
\qquad
\prod_{i,j} (1 + u_i v_j) = \sum_\bm s_\bm(\bu) s_{\bm'}(\bv),
\end{equation}
where $\bm'$ is the dual partition to $\bm$. In the fermionic case, since
$u_i^2 = v_j^2 = 0$, the two left-hand expressions coincide:

\begin{proposition}
Fix integers $n \geq 2$ and $k \geq 1$.
Given odd variables $u_i, v_j$ for $1 \leq i,j \leq n$,
\begin{equation}\label{Eprododd}
\prod_{i,j=1}^n \left( \frac{1}{1 - t u_i v_j} \right)^k =
\prod_{i,j=1}^n (1 + t u_i v_j)^k = 1 + k t (u_1 + \cdots + u_n)
(v_1 + \cdots + v_n).
\end{equation}
\end{proposition}

Notice the similarity to~\eqref{Epermodd}. In fact, a similar identity
holds for the more general product of factors $(f_0 \pm t f_1 u_i
v_j)^{\pm 1}$, and we leave the details to the interested reader.

\begin{proof}
As mentioned above, the first equality holds because all variables are
fermionic. We now prove the second equality; in doing so, note that all
terms $t u_i v_j$ are even, and hence commute pairwise. 
Since moreover all $u_i^2 = v_j^2 = 0$, hence proving the second
equality for $k=1$ implies it for higher $k$ by the binomial theorem.
Thus, we assume henceforth that $k=1$. Now viewing this product as a
polynomial in $t$, the constant term is $1$, and there are $2^{n^2}$
terms/monomials, each of which has $t$-degree at most $n^2$. In any
monomial of $t$-degree $> n$, the pigeonhole principle yields a factor of
a $u_i^2$ and a $v_j^2$, both of which vanish. Next, the linear terms in
$t$ clearly add up to $t (u_1 + \cdots + u_n) (v_1 + \cdots + v_n)$.

It remains to show that the coefficient of $t^k$ vanishes, for all $2
\leq k \leq n$. For convenience, we multiply the even
factors $(1 + u_i v_j)$ in lexicographic order $(1,1), (1,2), \dots,
(n,n)$. Then the coefficient of $t^k$ consists of terms of the form
\[
u_{i_1} v_{j_1} \cdots u_{i_k} v_{j_k} = (-1)^{\binom{k}{2}} u_{i_1}
\cdots u_{i_k} \cdot v_{j_1} \cdots v_{j_k},
\]
where $1 \leq i_1 < \cdots < i_k \leq n$ (since $u_i^2 = 0$ for all $i$)
and $j_1, \dots, j_k$ are pairwise distinct. It is now easy to see that
this term is obtained in multiple ways, where one can pair the tuple
$(j_1, j_2, \dots, j_k)$ with $(j_2, j_1, \dots, j_k)$ -- and this
procedure pairs off the terms into couples with opposite signs. Thus, the
sum of all of these terms vanishes. Proceeding in this fashion, all
quadratic and higher order terms in $t$ vanish, proving the result.
\end{proof}

Given the theory of symmetric functions, a natural follow-up
to these fermionic and bosonic Cauchy product identities is the
\textit{nonsymmetric} analogue of the bosonic identity, see \cite[Theorem
1.1 and Section 3]{Sahi}. We leave it to the interested reader to explore
if there exists a fermionic counterpart to \textit{loc.\ cit.}

\subsection{The case of $\varepsilon$-commuting sets of odd/even
variables}

In the preceding set of formulas, the two sets of variables $u_j, v_k$
were all odd/fermionic -- whereas they were all even/bosonic in an
earlier section. As a consequence, in both of these settings the
variables $x_{ij} = u_i v_j$ commute in both of these settings, which
makes the determinant well-defined regardless of how one expands it out.

In this concluding subsection (which is essentially an expanded remark),
we derive analogous identities  in a slightly more general setting. The
point is to draw attention to a parameter that is implicit in the
calculations in both of the above settings: the proportionality constant
$\varepsilon$ that one obtains when moving any $u$ past any $v$. In the
case of even/odd variables, we had specialized this parameter to equal
the scalar $\varepsilon = \pm 1$, respectively. However, the computations
in fact hold for arbitrary choice of $\varepsilon$ in \textit{either}
setting, because the power of this scalar merely keeps track of how many
$u$ move past how many $v$. Thus, similar to the variable $t$ that keeps
track of the common homogeneity degree in the $u$'s and the $v$'s
(separately), we now introduce another ``bookkeeping''
\textit{indeterminate} $\varepsilon$, which ends up keeping track of the
same information -- but now via the number of exchanges of $u$'s and
$v$'s. Notice, however, that the terms $x_{ij} = u_i v_j$ still pairwise
commute, so that $\det f[\bu \bv^T]$ stays well-defined.

Thus, we now write down the ``more general'' formulas in the above two
settings; the proofs are identical. In the case of \textbf{bosonic} $u_i$
and $v_j$, if moreover
\[
\varepsilon u_i v_j = v_j u_i, \quad \forall 1 \leq i,j \leq n,
\]
then given a subgroup $G \leq S_2 \wr S_n$ and a multiplicity-free
complex character $\chi$ of $G$, first fix an $R_\chi$-algebra $R$. Now
we work over the polynomial ring $R[\varepsilon]$, in the quotient of
$T^\bullet_{R[\varepsilon]}(U \oplus V)$ by the two-sided ideal generated
by
\[
\{ u_i \otimes u_j - u_j \otimes u_i, \ v_i \otimes v_j - v_j \otimes
v_i, \quad v_j \otimes u_i - \varepsilon \, u_i \otimes v_j \ : \ 1 \leq
i, j \leq n \};
\]
notice this ideal is ($S_2 \wr S_n) \times (S_2 \wr S_n)$-stable. Then
for any $f \in R[[t]]$ (or even $f \in R[\varepsilon][[t]]$),
\begin{equation}\label{Ebookkeeping1}
|G| E_\chi^\bu \cdot \prod_{i=1}^n f(t u_i v_i) =
|G| E_{\overline{\chi}}^\bv \cdot \prod_{i=1}^n f(t u_i v_i) = 
\sum_{{\bf m} \in \Z_{\geq 0}^n}
\varepsilon^{\binom{|\bm|}{2}} t^{|\bm|} f_{\bf m} \cdot
E_\chi^\bu (\bu^\bm) \cdot E_{\overline{\chi}}^\bv (\bv^\bm).
\end{equation}

Specializing to $G = S_n$ and $\chi$ the sign character (and canceling
$n!$ as in the discussion following~\eqref{Eperm}),
\begin{equation}
\det f[ t \bu \bv^T ] = \sum_{{\bf m} \in \Z_{\geq 0}^n, \
{\bf m} \text{decreasing}} \varepsilon^{\binom{|{\bf m}|}{2}}
t^{|{\bf m}| + \binom{n}{2}} \prod_{i=1}^n f_{m_i + n-i}
\cdot V(\bu) s_{\bf m}(\bu) \cdot V(\bv) s_{\bf m}(\bv).
\end{equation}

Similarly, in the case of \textbf{fermionic} $u_i$ and $v_j$, if moreover
$\varepsilon u_i v_j = v_j u_i\ \forall i,j$, the analogous formula is:
\begin{equation}\label{Ebookkeeping2}
|G| E_\chi^\bu \cdot \prod_{j=1}^n f(t u_j v_j) =
|G| E_{\overline{\chi}}^\bv \cdot \prod_{j=1}^n f(t u_j v_j) =
\sum_{J \subset [n]} \varepsilon^{\binom{|J|}{2}} f_0^{n-|J|} (f_1
t)^{|J|} \cdot E_\chi^\bu (\bu^J) \cdot E_{\overline{\chi}}^\bv (\bv^J)
\end{equation}

\noindent for arbitrary $G \leq S_2 \wr S_n$, any multiplicity-free
character $\chi$ of $G$, any $R_\chi$-algebra $R$, and any $f \in
R[[t]]$. Again specializing to $G = S_n$ and $\chi$ the sign character,
\begin{align}
&\ \det (f_0 + f_1 t u_i v_j)_{i,j=1}^n \notag\\
= &\ \varepsilon^{\binom{n}{2}} t^n n!  f_1^n \cdot
u_1 \cdots u_n \cdot v_1 \cdots v_n\\
& + \varepsilon^{\binom{n-1}{2}} t^{n-1} (n-1)! f_0 f_1^{n-1} \cdot
\sum_{i=1}^n (-1)^{i-1} u_1 \cdots \widehat{u_i} \cdots u_n \cdot
\sum_{j=1}^n (-1)^{j-1} v_1 \cdots \widehat{v_j} \cdots v_n.\notag
\end{align}

Similar formulas hold for the permanents, in both the bosonic and
fermionic settings. Moreover, the ``master
identities''~\eqref{Ebookkeeping1} and~\eqref{Ebookkeeping2} immediately
extend to arbitrary complex class functions $\psi$ of $G$, by linearity.

\section*{Acknowledgments}

We thank the referees for their detailed reading and valuable comments,
which led to a significant improvement of the paper.
A.K.\ was partially supported by Ramanujan Fellowship grant
SB/S2/RJN-121/2017, MATRICS grant MTR/2017/000295, and SwarnaJayanti
Fellowship grants SB/SJF/2019-20/14 and DST/SJF/MS/2019/3 from SERB and
DST (Govt.~of India), by grant F.510/25/CAS-II/2018(SAP-I) from UGC
(Govt.~of India), and by a Young Investigator Award from the Infosys
Foundation.
S.S. was partially supported by Simons Foundation grant 509766,
and NSF grants DMS-1939600 and DMS-2001537.



\begin{thebibliography}{88}
\bibitem{Amd}
T.~Amdeberhan.
\newblock A determinant of the Chudnovskys generalizing the elliptic
  Frobenius--Stickelberger--Cauchy determinantal identity.
\newblock
\href{http://www.combinatorics.org/ojs/index.php/eljc/article/view/v7i1n6}{\em
Electron. J. Combin.} 7, Note \#N6, 3 pp., 2000.

\bibitem{BGKP-fixeddim}
A.~Belton, D.~Guillot, A.~Khare, and M.~Putinar.
\newblock Matrix positivity preservers in fixed dimension. I.
\newblock \href{http://dx.doi.org/10.1016/j.aim.2016.04.016}{\em Adv.
  Math.}, 298:325--368, 2016.

\bibitem{Cauchy}
A.-L.~Cauchy.
\newblock M\'emoire sur les fonctions altern\'ees et sur les sommes
altern\'ees.
\newblock {\em Exercices Anal. et Phys. Math.}, 2:151--159, 1841.

\bibitem{Fr}
F.G.~Frobenius.
\newblock \"Uber die elliptischen Funktionen zweiter Art.
\newblock \href{http://dx.doi.org/10.1515/crll.1882.93.53}{\em J.\ reine
  angew.\ Math.}, 93:53--68, 1882.

\bibitem{FS}
F.G.~Frobenius and L.~Stickelberger.
\newblock Zur Theorie der elliptischen Functionen.
\newblock \href{http://dx.doi.org/10.1515/crll.1877.83.175}{\em J.\ reine
  angew.\ Math.}, 83:175--179, 1877.

\bibitem{horn}
R.A.~Horn.
\newblock The theory of infinitely divisible matrices and kernels.
\newblock \href{http://dx.doi.org/10.1090/S0002-9947-1969-0264736-5}{\em
  Trans.\ Amer.\ Math.\ Soc.}, 136:269--286, 1969.

\bibitem{IIO}
M.~Ishikawa, M.~Ito, and S.~Okada.
\newblock A compound determinant identity for rectangular matrices and
  determinants of Schur functions.
\newblock \href{http://dx.doi.org/10.1016/j.aam.2013.08.001}{\em Adv.
  Appl. Math.}, 51(5):635--654, 2013.

\bibitem{IOTZ}
M.~Ishikawa, S.~Okada, H.~Tagawa, and J.~Zeng.
\newblock Generalizations of Cauchy's determinant and Schur's Pfaffian.
\newblock \href{http://dx.doi.org/10.1016/j.aam.2005.07.001}{\em Adv.
  Appl. Math.}, 36(3):251--287, 2006.

\bibitem{Jacobi}
C.G.J.~Jacobi.
\newblock De functionibus alternantibus earumque divisione per productum
e differentiis elementorum conflatum.
\newblock \href{http://doi.org/10.1515/crll.1841.22.360}{\em J.\ reine
  angew. Math.}, 22:360--371, 1841.

\bibitem{horndet}
A.~Khare.
\newblock Smooth entrywise positivity preservers, a Horn--Loewner master
theorem, and symmetric function identities.
\newblock \href{http://dx.doi.org/10.1090/tran/8563}{\em Trans.\ Amer.
Math.\ Soc.}, 375(3):2217--2236, 2022.

\bibitem{KT}
A.~Khare and T.~Tao.
\newblock On the sign patterns of entrywise positivity preservers in
fixed dimension.
\newblock \href{http://dx.doi.org/10.1353/ajm.2021.0049}{\em Amer.\ J.\
Math.}, 143(6):1863--1929, 2021.

\bibitem{Kr1}
C.~Krattenthaler.
\newblock Advanced determinantal calculus.
\newblock
\href{http://www.mat.univie.ac.at/~slc/wpapers/s42kratt.html}{\em Sem.
  Lothar. Combin.}, 42, article B42q, 67 pp., 1998.

\bibitem{Kr2}
C.~Krattenthaler.
\newblock Advanced determinantal calculus: A complement.
\newblock \href{http://dx.doi.org/10.1016/j.laa.2005.06.042}{\em Linear
  Algebra Appl.}, 411:68--166, 2005.

\bibitem{Ku}
G.~Kuperberg.
\newblock Symmetry classes of alternating-sign matrices under one roof.
\newblock \href{http://dx.doi.org/10.2307/3597283}{\em Ann. of Math.}
  156(3):835--866, 2002.

\bibitem{LLT}
D.~Laksov, A.~Lascoux, and A.~Thorup.
\newblock On Giambelli's theorem for complete correlations.
\newblock \href{http://dx.doi.org/10.1007/BF02392836}{\em Acta Math.},
  162:143--199, 1989.

\bibitem{LR}
D.E.~Littlewood and A.R.~Richardson.
\newblock Group characters and algebras.
\newblock \href{http://dx.doi.org/10.1098/rsta.1934.0015}{\em Phil.\
  Trans.\ Royal Soc.\ A}, 233(721--730), 99--124.

\bibitem{Macdonald}
I.G.~Macdonald.
\newblock {\em Symmetric functions and {H}all polynomials}.
\newblock
\href{https://global.oup.com/academic/product/symmetric-functions-and-hall-polynomials-9780198739128}{Oxford
Mathematical Monographs}. The Clarendon Press, Oxford University Press,
New York, second edition, 1995.
\newblock With contributions by A. Zelevinsky, Oxford Science
  Publications.


\bibitem{RS}
H.~Rosengren and M.~Schlosser.
\newblock Elliptic determinant evaluations and the Macdonald
  identities for affine root systems.
\newblock \href{http://dx.doi.org/10.1112/S0010437X0600203X}{\em Compos.\
  Math.} 142:937--961, 2006.

\bibitem{Rudin59}
W.~Rudin.
\newblock Positive definite sequences and absolutely monotonic functions.
\newblock \href{http://dx.doi.org/10.1215/S0012-7094-59-02659-6}{\em Duke
  Math. J.}, 26(4):617--622, 1959.

\bibitem{Sahi}
S.~Sahi.
\newblock A new scalar product for nonsymmetric Jack polynomials.
\newblock \href{http://dx.doi.org/10.1155/S107379289600061X}{\em Int.\
Math.\ Res.\ Not.\ IMRN}, 1996(20):997--1004, 1996.

\bibitem{Schoenberg42}
I.J.~Schoenberg.
\newblock Positive definite functions on spheres.
\newblock \href{http://dx.doi.org/10.1215/S0012-7094-42-00908-6}{\em Duke
  Math. J.}, 9(1):96--108, 1942.

\bibitem{Schur1911}
I.~Schur.
\newblock {B}emerkungen zur {T}heorie der beschr{\"a}nkten
  {B}ilinearformen mit unendlich vielen {V}er{\"a}nderlichen.
\newblock \href{http://dx.doi.org/10.1515/crll.1911.140.1}{\em J. reine
  angew. Math.}, 140:1--28, 1911.

\bibitem{Springer}
T.A.~Springer.
\newblock A construction of representations of {W}eyl groups.
\newblock \href{http://dx.doi.org/10.1007/BF01403165}{\em Invent.\
Math.}, 44(3):279--293, 1978.

\end{thebibliography}


\end{document}